\documentclass[12pt,a4paper]{amsart}
\usepackage{amsmath,amsthm,amsfonts,amssymb,mathrsfs}
\usepackage{subfig} 
\usepackage{graphicx}
\usepackage{color}
\usepackage{ulem}

\newcommand\R{\mathbb{R}}

\renewcommand\L{\mathcal{L}}

\def\u{\bar{u}}
\def\v{\bar{v}}

\def\bxi{{\bar{\xi}}}

\newtheorem{theorem}{Theorem}
\newtheorem{cor}[theorem]{Corollary}

\newtheorem{lemma}[theorem]{Lemma}
\newtheorem{prop}[theorem]{Proposition}

\numberwithin{equation}{section}
\numberwithin{theorem}{section}
\numberwithin{figure}{section}

\theoremstyle{remark}
\newtheorem{rem}[theorem]{Remark}

\usepackage{delarray}
\usepackage{enumerate}

\begin{document}

\title[Nonlocal model of pattern formation]{Dynamical spike solutions in a nonlocal model \\ of pattern formation}

\author[A.~Marciniak-Czochra]{Anna Marciniak-Czochra}
\address[A. Marciniak-Czochra]{
Institute of Applied Mathematics,  Interdisciplinary Center for Scientific Computing (IWR) and BIOQUANT, University of Heidelberg, 69120 Heidelberg, Germany}
\email{anna.marciniak@iwr.uni-heidelberg.de}
\urladdr {http://www.biostruct.uni-hd.de/}

\author[S.~H\"arting]{Steffen H\"arting}
\address[S. H\"arting]{
Institute of Applied Mathematics,  University of Heidelberg, 69120 Heidelberg, Germany}
\email{steffen.haerting@bioquant.uni-heidelberg.de}

\author[G.~Karch]{Grzegorz Karch}
\address[G. Karch]{
 Instytut Matematyczny, Uniwersytet Wroc\l awski,
 pl. Grunwaldzki 2/4, 50-384 Wroc\-\l aw, Poland}
\email{grzegorz.karch@math.uni.wroc.pl}
\urladdr {http://www.math.uni.wroc.pl/~karch}

\author[K.~Suzuki]{Kanako Suzuki}
\address[K. Suzuki]{
College of Science, Ibaraki University,
2-1-1 Bunkyo, Mito 310-8512, Japan}
\email{kanako.suzuki.sci2@vc.ibaraki.ac.jp}

\date{\today}


\begin{abstract}
%
%
%
Coupling a reaction-diffusion
equation with ordinary differential equations (ODE) may lead to diffusion-driven instability (DDI) which, in contrast to the classical reaction-diffusion models, causes destabilization of both, constant
solutions and Turing patterns. Using a shadow-type  limit of a reaction-diffusion-ODE model, we show that in such cases the instability driven by  nonlocal terms (a counterpart of DDI) may lead to formation of unbounded 
 spike patterns.


\end{abstract}
\maketitle

\section{Introduction}
\subsection{Reaction-diffusion-ODE systems}

Classical models of pattern formation are based on {\it diffusion-driven
instability} (DDI, Turing instability) of constant stationary solutions of reaction-diffusion equations,
which leads to emergence of stable Turing patterns formed
around that equilibrium. The emerging patterns can be spatially monotone or spatially periodic. Their shape depends on the ratio of diffusion coefficients as well as on a scaling coefficient which reflects the relationship between the diffusion coefficients and  the domain size \cite{Murray2}. 

Interestingly, a variety of possible patterns increases when setting the smaller diffusion coefficient to zero, {\it i.e.}~considering a single reaction-diffusion equation coupled to an ordinary differential equation (ODE)
 \begin{equation}\label{RD}
u_t  =   f(u,v), \qquad v_t  =   D \Delta v+g(u,v),
\end{equation}
supplemented with the Neumann boundary condition for $v=v(x,t)$.

Such models arise from coupling of the diffusive processes with processes which are localized in space, such as, for example, growth processes \cite{MCK06,MCK07,MCK08,Pham} or intracellular signaling  \cite{Hock,Klika,MC03,USOO}. In the latter case, macroscopic reaction-diffusion-ODE models have been derived as a homogenization limit of the models describing coupling of cell-localized processes with a cell-to-cell communication through a diffusion in a cell assembly \cite{MCP,MC12}. 

The dynamics of such models appears to be very different from that of classical reaction-diffusion models. Systems coupling a single reaction-diffusion equation with ODEs may exhibit DDI as shown in \cite{MC03} and discussed more recently on several examples from mathematical biology in \cite{Klika}. However, in this case all Turing patterns are unstable, {\it i.e.}~the same mechanism which destabilizes constant solutions, destabilizes also all continuous spatially heterogenous stationary solutions  \cite{MKS12,MKS16,LMCTW}.  The question then arises as to which patterns, if any, can be exhibited in such models. Two scenarios have been observed numerically: either a convergence to stable stationary patterns with jump-discontinuity \cite{HMCT16,Hock} or 
  an emergence of  multimodal dynamical spike structures  \cite{HMC13,Pham}.

In the first case, solutions of the model are uniformly bounded and they converge to a {\it far from equilibrium} patterns with jump discontinuity which results from the existence of multiple quasi-stationary solutions in the ODE subsystems. The hysteresis  in the location of  stable branches of those solutions allows 
the authors of Ref.~\cite{HMCT16}
to construct a continuum of stationary solutions with jump discontinuity, which may be either monotone or periodic or
irregular. 
In Ref.~\cite{HMCT16}, conditions for linear stability of such patterns in a  topology excluding the discontinuity points  have been provided. Then, the emergence and stability of the patterns with jump discontinuity has been proved for a receptor-ligand model with DDI and hysteresis.

\subsection{Problem formulation}

The goal of this paper is to understand the ``spiky'' unbounded  patterns  emerging in the second class of models \cite{HMC13,MCK07,MCK08,Pham}. We consider this problem in a particular case  of  a model coupling cell growth with receptor-ligand dynamics, which was proposed in \cite{MCK07,MCK08} as a three-equation model and reduced to a reaction-diffusion-ODE system of the form \eqref{RD} in \cite{HMC13,MKS12}. 

So far, analysis of the model showed instability of all stationary solutions, including the solutions with jump-discontinuity \cite{MKS12,MKS16}, while model simulations indicated emergence of dynamical spike patterns \cite{HMC13}. To understand the underlying dynamics, we propose to further reduce the model and focus on its shadow-limit. It results in the following nonlocal system with two unknown nonnegative functions
$u=u(x,t)$ and $\xi=\xi(t)$, 
 describing an evolution of a density of growing cells and a uniformly distributed  growth factor, respectively,
\begin{align}
u_t&=\Big(\frac{a u\xi}{1+u\xi} -d\Big) u &\text{for}& \quad x\in \Omega, \; t>0, \label{eqsc1}\\
\xi_t &=  - \xi - \xi \int_\Omega u^2 \;dx  +\kappa_0 &
\text{for}&\quad   t>0. \label{eqsc2}
\end{align}
Here, $a,d,\kappa_0$ are positive constants,
 $\Omega\subset \R^n$ is an arbitrary bounded measurable set of a finite nonzero measure and we  assume that $|\Omega|=1$, without loss of generality. 
Moreover, we supplement  equations \eqref{eqsc1}-\eqref{eqsc2}
with nonnegative initial conditions
\begin{equation}\label{inisc}
u(0,x)=u_0(x), \qquad 
\xi(0) = \xi_0.
\end{equation}

The reduction of a reaction-diffusion model to a nonlocal problem can be obtained  after passing with the diffusion coefficient $D$ in second equation
 to the limit $D\to \infty $, see Theorem \ref{thm:A} in Appendix  \ref{ap:B} for a rigorous
proof of this claim. Shadow-limit has attracted a considerable interest in the literature in the case where the first equation is a quasilinear parabolic equation, starting from the papers of Keener \cite{Keener78}, and Hale and Sakamoto \cite{HaleSakamoto89}. An approach using semigroup convergence has been recently undertaken by Bobrowski \cite{Bobrowski}. Here, we apply the shadow limit to reduce a reaction-diffusion-ODE model  \eqref{RD} to a system of integro-differential equations \eqref{eqsc1}-\eqref{inisc}.  

\subsection{Content of the paper} In this work we focus on the asymptotic behavior of the solutions of problem problem \eqref{eqsc1}--\eqref{inisc}. 

In the first steps, we show  the following results on the nonlocal initial value problem \eqref{eqsc1}-\eqref{inisc} where,  to streamline the
 analysis, we assume that $u_0 \in C (\overline\Omega)$ as well as $u_0(x)\geq 0$ for all $x\in \overline\Omega$ and $\xi_0\geq 0$.
\begin{itemize}
\item Problem  \eqref{eqsc1}--\eqref{inisc} has a unique global-in-time nonnegative solution for every nonnegative initial condition \eqref{inisc}, see Proposition 
\ref{thm:existence:carc}, below.
\item The ``total mass'' $\int_\Omega u(x,t)\,dx$ of every nonnegative solution 
to problem \eqref{eqsc1}--\eqref{inisc}
is bounded on $[0,\infty)$ (Proposition 
\ref{thm:existence:carc}).
\item Each $x$-independent solution of problem 
\eqref{eqsc1}--\eqref{inisc} is bounded on the half-line $[0,\infty)$,
see Proposition \ref{thm:kinetic}.
\item {\it All} stationary solutions of problem  \eqref{eqsc1}--\eqref{inisc} are unstable (Section \ref{sec:stat}).
\end{itemize}

Our main result (Theorem~\ref{thm:growth}) shows that the nonlocal coupling may lead to a loss of boundedness of solutions. More specifically, we prove that although space homogeneous solutions of the model are uniformly bounded in time, there exist space heterogeneous solutions with an unbounded growth as $t\to \infty$, as depicted in Fig.~\ref{figure}.


{

\begin{figure}
\begin{tabular}{p{18em}p{18em}}
  {\small
    \centering
\begingroup
  \makeatletter
  \providecommand\color[2][]{%
    \GenericError{(gnuplot) \space\space\space\@spaces}{%
      Package color not loaded in conjunction with
      terminal option `colourtext'%
    }{See the gnuplot documentation for explanation.%
    }{Either use 'blacktext' in gnuplot or load the package
      color.sty in LaTeX.}%
    \renewcommand\color[2][]{}%
  }%
  \providecommand\includegraphics[2][]{%
    \GenericError{(gnuplot) \space\space\space\@spaces}{%
      Package graphicx or graphics not loaded%
    }{See the gnuplot documentation for explanation.%
    }{The gnuplot epslatex terminal needs graphicx.sty or graphics.sty.}%
    \renewcommand\includegraphics[2][]{}%
  }%
  \providecommand\rotatebox[2]{#2}%
  \@ifundefined{ifGPcolor}{%
    \newif\ifGPcolor
    \GPcolorfalse
  }{}%
  \@ifundefined{ifGPblacktext}{%
    \newif\ifGPblacktext
    \GPblacktexttrue
  }{}%
  \let\gplgaddtomacro\g@addto@macro
  \gdef\gplbacktext{}%
  \gdef\gplfronttext{}%
  \makeatother
  \ifGPblacktext
    \def\colorrgb#1{}%
    \def\colorgray#1{}%
  \else
    \ifGPcolor
      \def\colorrgb#1{\color[rgb]{#1}}%
      \def\colorgray#1{\color[gray]{#1}}%
      \expandafter\def\csname LTw\endcsname{\color{white}}%
      \expandafter\def\csname LTb\endcsname{\color{black}}%
      \expandafter\def\csname LTa\endcsname{\color{black}}%
      \expandafter\def\csname LT0\endcsname{\color[rgb]{1,0,0}}%
      \expandafter\def\csname LT1\endcsname{\color[rgb]{0,1,0}}%
      \expandafter\def\csname LT2\endcsname{\color[rgb]{0,0,1}}%
      \expandafter\def\csname LT3\endcsname{\color[rgb]{1,0,1}}%
      \expandafter\def\csname LT4\endcsname{\color[rgb]{0,1,1}}%
      \expandafter\def\csname LT5\endcsname{\color[rgb]{1,1,0}}%
      \expandafter\def\csname LT6\endcsname{\color[rgb]{0,0,0}}%
      \expandafter\def\csname LT7\endcsname{\color[rgb]{1,0.3,0}}%
      \expandafter\def\csname LT8\endcsname{\color[rgb]{0.5,0.5,0.5}}%
    \else
      \def\colorrgb#1{\color{black}}%
      \def\colorgray#1{\color[gray]{#1}}%
      \expandafter\def\csname LTw\endcsname{\color{white}}%
      \expandafter\def\csname LTb\endcsname{\color{black}}%
      \expandafter\def\csname LTa\endcsname{\color{black}}%
      \expandafter\def\csname LT0\endcsname{\color{black}}%
      \expandafter\def\csname LT1\endcsname{\color{black}}%
      \expandafter\def\csname LT2\endcsname{\color{black}}%
      \expandafter\def\csname LT3\endcsname{\color{black}}%
      \expandafter\def\csname LT4\endcsname{\color{black}}%
      \expandafter\def\csname LT5\endcsname{\color{black}}%
      \expandafter\def\csname LT6\endcsname{\color{black}}%
      \expandafter\def\csname LT7\endcsname{\color{black}}%
      \expandafter\def\csname LT8\endcsname{\color{black}}%
    \fi
  \fi
  \setlength{\unitlength}{0.0500bp}%
  \begin{picture}(3600.00,3528.00)%
    \gplgaddtomacro\gplbacktext{%
    }%
    \gplgaddtomacro\gplfronttext{%
      \csname LTb\endcsname%
      \put(507,1142){\makebox(0,0){\strut{} 0}}%
      \put(1102,835){\makebox(0,0){\strut{} 10}}%
      \put(1697,567){\makebox(0,0){\strut{} 20}}%
      \put(908,779){\makebox(0,0){\strut{}$t$}}%
      \put(1893,615){\makebox(0,0){\strut{} 0}}%
      \put(3083,1150){\makebox(0,0){\strut{} 1}}%
      \put(2692,819){\makebox(0,0){\strut{}$x$}}%
      \put(484,1656){\makebox(0,0)[r]{\strut{} 0}}%
      \put(484,1905){\makebox(0,0)[r]{\strut{} 40}}%
      \put(484,2154){\makebox(0,0)[r]{\strut{} 80}}%
      \put(484,2404){\makebox(0,0)[r]{\strut{} 120}}%
    }%
    \gplbacktext
    \put(0,0){\includegraphics{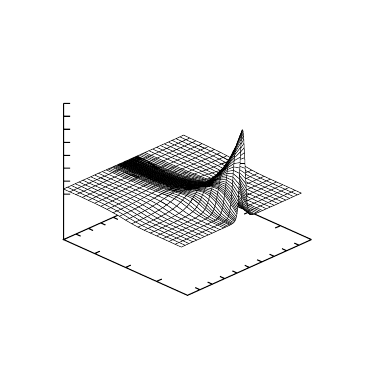}}%
    \gplfronttext
  \end{picture}%
\endgroup
}

&

{\small
  \centering
\begingroup
  \makeatletter
  \providecommand\color[2][]{%
    \GenericError{(gnuplot) \space\space\space\@spaces}{%
      Package color not loaded in conjunction with
      terminal option `colourtext'%
    }{See the gnuplot documentation for explanation.%
    }{Either use 'blacktext' in gnuplot or load the package
      color.sty in LaTeX.}%
    \renewcommand\color[2][]{}%
  }%
  \providecommand\includegraphics[2][]{%
    \GenericError{(gnuplot) \space\space\space\@spaces}{%
      Package graphicx or graphics not loaded%
    }{See the gnuplot documentation for explanation.%
    }{The gnuplot epslatex terminal needs graphicx.sty or graphics.sty.}%
    \renewcommand\includegraphics[2][]{}%
  }%
  \providecommand\rotatebox[2]{#2}%
  \@ifundefined{ifGPcolor}{%
    \newif\ifGPcolor
    \GPcolorfalse
  }{}%
  \@ifundefined{ifGPblacktext}{%
    \newif\ifGPblacktext
    \GPblacktexttrue
  }{}%
  \let\gplgaddtomacro\g@addto@macro
  \gdef\gplbacktext{}%
  \gdef\gplfronttext{}%
  \makeatother
  \ifGPblacktext
    \def\colorrgb#1{}%
    \def\colorgray#1{}%
  \else
    \ifGPcolor
      \def\colorrgb#1{\color[rgb]{#1}}%
      \def\colorgray#1{\color[gray]{#1}}%
      \expandafter\def\csname LTw\endcsname{\color{white}}%
      \expandafter\def\csname LTb\endcsname{\color{black}}%
      \expandafter\def\csname LTa\endcsname{\color{black}}%
      \expandafter\def\csname LT0\endcsname{\color[rgb]{1,0,0}}%
      \expandafter\def\csname LT1\endcsname{\color[rgb]{0,1,0}}%
      \expandafter\def\csname LT2\endcsname{\color[rgb]{0,0,1}}%
      \expandafter\def\csname LT3\endcsname{\color[rgb]{1,0,1}}%
      \expandafter\def\csname LT4\endcsname{\color[rgb]{0,1,1}}%
      \expandafter\def\csname LT5\endcsname{\color[rgb]{1,1,0}}%
      \expandafter\def\csname LT6\endcsname{\color[rgb]{0,0,0}}%
      \expandafter\def\csname LT7\endcsname{\color[rgb]{1,0.3,0}}%
      \expandafter\def\csname LT8\endcsname{\color[rgb]{0.5,0.5,0.5}}%
    \else
      \def\colorrgb#1{\color{black}}%
      \def\colorgray#1{\color[gray]{#1}}%
      \expandafter\def\csname LTw\endcsname{\color{white}}%
      \expandafter\def\csname LTb\endcsname{\color{black}}%
      \expandafter\def\csname LTa\endcsname{\color{black}}%
      \expandafter\def\csname LT0\endcsname{\color{black}}%
      \expandafter\def\csname LT1\endcsname{\color{black}}%
      \expandafter\def\csname LT2\endcsname{\color{black}}%
      \expandafter\def\csname LT3\endcsname{\color{black}}%
      \expandafter\def\csname LT4\endcsname{\color{black}}%
      \expandafter\def\csname LT5\endcsname{\color{black}}%
      \expandafter\def\csname LT6\endcsname{\color{black}}%
      \expandafter\def\csname LT7\endcsname{\color{black}}%
      \expandafter\def\csname LT8\endcsname{\color{black}}%
    \fi
  \fi
  \setlength{\unitlength}{0.0500bp}%
  \begin{picture}(3600.00,3528.00)%
    \gplgaddtomacro\gplbacktext{%
    }%
    \gplgaddtomacro\gplfronttext{%
      \csname LTb\endcsname%
      \put(507,1142){\makebox(0,0){\strut{} 0}}%
      \put(824,959){\makebox(0,0){\strut{} 4}}%
      \put(1141,817){\makebox(0,0){\strut{} 8}}%
      \put(1459,674){\makebox(0,0){\strut{} 12}}%
      \put(908,779){\makebox(0,0){\strut{}$t$}}%
      \put(1893,615){\makebox(0,0){\strut{} 0}}%
      \put(3083,1150){\makebox(0,0){\strut{} 1}}%
      \put(2692,819){\makebox(0,0){\strut{}$x$}}%
      \put(484,1656){\makebox(0,0)[r]{\strut{} 0}}%
      \put(484,1946){\makebox(0,0)[r]{\strut{} 20}}%
      \put(484,2237){\makebox(0,0)[r]{\strut{} 40}}%
      \put(484,2528){\makebox(0,0)[r]{\strut{} 60}}%
    }%
    \gplbacktext
    \put(0,0){\includegraphics{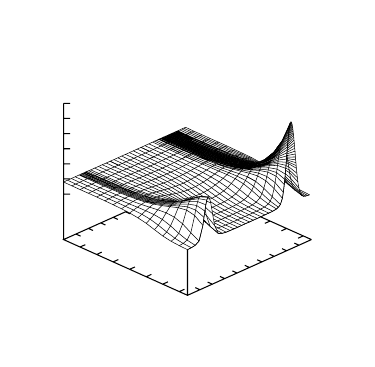}}%
    \gplfronttext
  \end{picture}%
\endgroup
}
\\
{\small
  \centering
\begingroup
  \makeatletter
  \providecommand\color[2][]{%
    \GenericError{(gnuplot) \space\space\space\@spaces}{%
      Package color not loaded in conjunction with
      terminal option `colourtext'%
    }{See the gnuplot documentation for explanation.%
    }{Either use 'blacktext' in gnuplot or load the package
      color.sty in LaTeX.}%
    \renewcommand\color[2][]{}%
  }%
  \providecommand\includegraphics[2][]{%
    \GenericError{(gnuplot) \space\space\space\@spaces}{%
      Package graphicx or graphics not loaded%
    }{See the gnuplot documentation for explanation.%
    }{The gnuplot epslatex terminal needs graphicx.sty or graphics.sty.}%
    \renewcommand\includegraphics[2][]{}%
  }%
  \providecommand\rotatebox[2]{#2}%
  \@ifundefined{ifGPcolor}{%
    \newif\ifGPcolor
    \GPcolorfalse
  }{}%
  \@ifundefined{ifGPblacktext}{%
    \newif\ifGPblacktext
    \GPblacktexttrue
  }{}%
  \let\gplgaddtomacro\g@addto@macro
  \gdef\gplbacktext{}%
  \gdef\gplfronttext{}%
  \makeatother
  \ifGPblacktext
    \def\colorrgb#1{}%
    \def\colorgray#1{}%
  \else
    \ifGPcolor
      \def\colorrgb#1{\color[rgb]{#1}}%
      \def\colorgray#1{\color[gray]{#1}}%
      \expandafter\def\csname LTw\endcsname{\color{white}}%
      \expandafter\def\csname LTb\endcsname{\color{black}}%
      \expandafter\def\csname LTa\endcsname{\color{black}}%
      \expandafter\def\csname LT0\endcsname{\color[rgb]{1,0,0}}%
      \expandafter\def\csname LT1\endcsname{\color[rgb]{0,1,0}}%
      \expandafter\def\csname LT2\endcsname{\color[rgb]{0,0,1}}%
      \expandafter\def\csname LT3\endcsname{\color[rgb]{1,0,1}}%
      \expandafter\def\csname LT4\endcsname{\color[rgb]{0,1,1}}%
      \expandafter\def\csname LT5\endcsname{\color[rgb]{1,1,0}}%
      \expandafter\def\csname LT6\endcsname{\color[rgb]{0,0,0}}%
      \expandafter\def\csname LT7\endcsname{\color[rgb]{1,0.3,0}}%
      \expandafter\def\csname LT8\endcsname{\color[rgb]{0.5,0.5,0.5}}%
    \else
      \def\colorrgb#1{\color{black}}%
      \def\colorgray#1{\color[gray]{#1}}%
      \expandafter\def\csname LTw\endcsname{\color{white}}%
      \expandafter\def\csname LTb\endcsname{\color{black}}%
      \expandafter\def\csname LTa\endcsname{\color{black}}%
      \expandafter\def\csname LT0\endcsname{\color{black}}%
      \expandafter\def\csname LT1\endcsname{\color{black}}%
      \expandafter\def\csname LT2\endcsname{\color{black}}%
      \expandafter\def\csname LT3\endcsname{\color{black}}%
      \expandafter\def\csname LT4\endcsname{\color{black}}%
      \expandafter\def\csname LT5\endcsname{\color{black}}%
      \expandafter\def\csname LT6\endcsname{\color{black}}%
      \expandafter\def\csname LT7\endcsname{\color{black}}%
      \expandafter\def\csname LT8\endcsname{\color{black}}%
    \fi
  \fi
  \setlength{\unitlength}{0.0500bp}%
  \begin{picture}(3600.00,3528.00)%
    \gplgaddtomacro\gplbacktext{%
    }%
    \gplgaddtomacro\gplfronttext{%
      \csname LTb\endcsname%
      \put(507,1142){\makebox(0,0){\strut{} 0}}%
      \put(824,959){\makebox(0,0){\strut{} 4}}%
      \put(1141,817){\makebox(0,0){\strut{} 8}}%
      \put(1459,674){\makebox(0,0){\strut{} 12}}%
      \put(908,779){\makebox(0,0){\strut{}$t$}}%
      \put(1893,615){\makebox(0,0){\strut{} 0}}%
      \put(3083,1150){\makebox(0,0){\strut{} 1}}%
      \put(2692,819){\makebox(0,0){\strut{}$x$}}%
      \put(484,1656){\makebox(0,0)[r]{\strut{} 0}}%
      \put(484,1850){\makebox(0,0)[r]{\strut{} 10}}%
      \put(484,2043){\makebox(0,0)[r]{\strut{} 20}}%
      \put(484,2237){\makebox(0,0)[r]{\strut{} 30}}%
      \put(484,2431){\makebox(0,0)[r]{\strut{} 40}}%
    }%
    \gplbacktext
    \put(0,0){\includegraphics{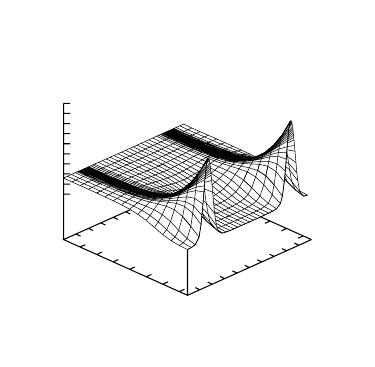}}%
    \gplfronttext
  \end{picture}%
\endgroup
}

&

{\small
  \centering
\begingroup
  \makeatletter
  \providecommand\color[2][]{%
    \GenericError{(gnuplot) \space\space\space\@spaces}{%
      Package color not loaded in conjunction with
      terminal option `colourtext'%
    }{See the gnuplot documentation for explanation.%
    }{Either use 'blacktext' in gnuplot or load the package
      color.sty in LaTeX.}%
    \renewcommand\color[2][]{}%
  }%
  \providecommand\includegraphics[2][]{%
    \GenericError{(gnuplot) \space\space\space\@spaces}{%
      Package graphicx or graphics not loaded%
    }{See the gnuplot documentation for explanation.%
    }{The gnuplot epslatex terminal needs graphicx.sty or graphics.sty.}%
    \renewcommand\includegraphics[2][]{}%
  }%
  \providecommand\rotatebox[2]{#2}%
  \@ifundefined{ifGPcolor}{%
    \newif\ifGPcolor
    \GPcolorfalse
  }{}%
  \@ifundefined{ifGPblacktext}{%
    \newif\ifGPblacktext
    \GPblacktexttrue
  }{}%
  \let\gplgaddtomacro\g@addto@macro
  \gdef\gplbacktext{}%
  \gdef\gplfronttext{}%
  \makeatother
  \ifGPblacktext
    \def\colorrgb#1{}%
    \def\colorgray#1{}%
  \else
    \ifGPcolor
      \def\colorrgb#1{\color[rgb]{#1}}%
      \def\colorgray#1{\color[gray]{#1}}%
      \expandafter\def\csname LTw\endcsname{\color{white}}%
      \expandafter\def\csname LTb\endcsname{\color{black}}%
      \expandafter\def\csname LTa\endcsname{\color{black}}%
      \expandafter\def\csname LT0\endcsname{\color[rgb]{1,0,0}}%
      \expandafter\def\csname LT1\endcsname{\color[rgb]{0,1,0}}%
      \expandafter\def\csname LT2\endcsname{\color[rgb]{0,0,1}}%
      \expandafter\def\csname LT3\endcsname{\color[rgb]{1,0,1}}%
      \expandafter\def\csname LT4\endcsname{\color[rgb]{0,1,1}}%
      \expandafter\def\csname LT5\endcsname{\color[rgb]{1,1,0}}%
      \expandafter\def\csname LT6\endcsname{\color[rgb]{0,0,0}}%
      \expandafter\def\csname LT7\endcsname{\color[rgb]{1,0.3,0}}%
      \expandafter\def\csname LT8\endcsname{\color[rgb]{0.5,0.5,0.5}}%
    \else
      \def\colorrgb#1{\color{black}}%
      \def\colorgray#1{\color[gray]{#1}}%
      \expandafter\def\csname LTw\endcsname{\color{white}}%
      \expandafter\def\csname LTb\endcsname{\color{black}}%
      \expandafter\def\csname LTa\endcsname{\color{black}}%
      \expandafter\def\csname LT0\endcsname{\color{black}}%
      \expandafter\def\csname LT1\endcsname{\color{black}}%
      \expandafter\def\csname LT2\endcsname{\color{black}}%
      \expandafter\def\csname LT3\endcsname{\color{black}}%
      \expandafter\def\csname LT4\endcsname{\color{black}}%
      \expandafter\def\csname LT5\endcsname{\color{black}}%
      \expandafter\def\csname LT6\endcsname{\color{black}}%
      \expandafter\def\csname LT7\endcsname{\color{black}}%
      \expandafter\def\csname LT8\endcsname{\color{black}}%
    \fi
  \fi
  \setlength{\unitlength}{0.0500bp}%
  \begin{picture}(3600.00,3528.00)%
    \gplgaddtomacro\gplbacktext{%
    }%
    \gplgaddtomacro\gplfronttext{%
      \csname LTb\endcsname%
      \put(507,1142){\makebox(0,0){\strut{} 0}}%
      \put(824,959){\makebox(0,0){\strut{} 4}}%
      \put(1141,817){\makebox(0,0){\strut{} 8}}%
      \put(1459,674){\makebox(0,0){\strut{} 12}}%
      \put(908,779){\makebox(0,0){\strut{}$t$}}%
      \put(1893,615){\makebox(0,0){\strut{} 0}}%
      \put(3083,1150){\makebox(0,0){\strut{} 1}}%
      \put(2692,819){\makebox(0,0){\strut{}$x$}}%
      \put(484,1656){\makebox(0,0)[r]{\strut{} 0}}%
      \put(484,1874){\makebox(0,0)[r]{\strut{} 4}}%
      \put(484,2092){\makebox(0,0)[r]{\strut{} 8}}%
      \put(484,2310){\makebox(0,0)[r]{\strut{} 12}}%
      \put(484,2528){\makebox(0,0)[r]{\strut{} 16}}%
    }%
    \gplbacktext
    \put(0,0){\includegraphics{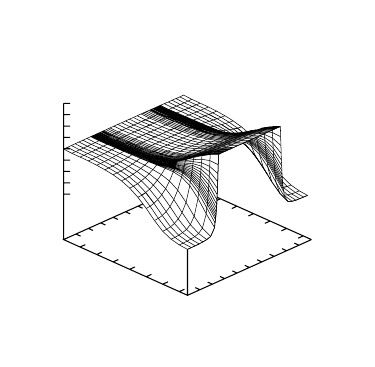}}%
    \gplfronttext
  \end{picture}%
\endgroup
}
\end{tabular}
\caption{Numerical simulations of solutions to 
the nonlocal  system  \eqref{eqsc1}-\eqref{eqsc2} supplemented with different initial conditions, see Remark \ref{rem:fig} for more details.}
\label{figure}
\end{figure}

}



\begin{rem}[Numerical simulations of the model]\label{rem:fig}
Our mathematical results on a growth of unbounded patterns 
 are motivated by numerical simulations of solutions to 
the nonlocal  system  \eqref{eqsc1}-\eqref{inisc} which we present in 
Fig.~\ref{figure}.
Each graph of  Fig.~\ref{figure} 
shows   the function $u=u(x,t)$  on the interval $\Omega=[0,1]$
which is a solution (togeher with a suitable function $\xi=\xi(t)$)
to system \eqref{eqsc1}-\eqref{eqsc2}
 with the parameters  $a=2$, $d=1$, $\kappa_0 =65/8$. 
Note that, in this case, 
 the constant vector $(\u,\bxi)=(8,1/8)$ is an asymptotically
 stable solution of the corresponding ODE system \eqref{eq-kin}
({\it cf.} \cite[Theorem B.2]{MKS16} and Remark~\ref{rem:DDI:carc}).
Thus, each figure shows a destabilization 
of  the spatially homogeneous steady state $(\u,\bxi)=(8,1/8)$
of the nonlocal system  \eqref{eqsc1}-\eqref{inisc} due to the Turing instability.
In particular,  in Fig.~\ref{figure} we can observe the following unbounded patterns:
\begin{itemize}
\item
{\it Single spike pattern.} For   the initial data
 $u_0(x) = 8- 0.05  (\cos(2 \pi x)+ 0.25 (1-x))$ 
(with one global maximum)
and $\xi_0 = 1/8$, we observe a formation of an unbounded  single spike at  the point of the global 
maximum  
of the initial function $u_0(x)$.

\item
 {\it Spike competition.} For   the initial data
   $u_0(x) = 8+ 0.05  \sin(3 \pi x) (1+0.1 x)$ 
(with two local maxima of different values)
and $\xi_0 = 1/8$,
we observe that initially, after the destabilization of the spatially homogeneous steady state $(\u,\bxi)=(8,1/8)$, a pattern with two peaks develops. 
The spikes are located at 
two points of 
 the local maxima of the initial function $u_0$. However,
when time is elapsing,
the spike corresponding to a larger initial value persists and growths to $+\infty$, 
while the other one decays.

\item
 {\it Double spike pattern.} For the 
the initial data 
 $u_0(x) = 8+ 0.05 \sin(3 \pi x) $ 
(with two local maxima of the same value)
and $\xi_0 = 1/8$,
 we observe Turing instability of the spatially homogeneous steady state  $(\u,\bxi)=(8,1/8)$ 
leading to formation of two spikes at two points of 
 the local maxima of the initial function $u_0$.

 \item
 {\it  Plateau pattern.}  For the  initial data 
  $ u_0(x) = 8- 0.05 \sin(2 \pi x + 0.5 \pi)$ for $x \notin (0.25,0.75)$ and 
$u_0(x) = 8$ for $x \in [0.25,0.75]$
(a maximal value achieved at an interval) 
 and $\xi_0 = 1/8$,  we observe that
 the Turing instability of the spatially homogeneous steady state  $(\u,\bxi)=(8,1/8)$ 
 leads to destabilization of this constant equilibrium and the solution converges to a 
 bounded, discontinuous pattern,  see Remark \ref{rem:plateau} below for more explanation.

\end{itemize}

\end{rem}

Mathematical results reported in this paper combined with  
 numerical investigations of model \eqref{eqsc1}-\eqref{inisc}
  led us to a conjecture that the model solutions tend asymptotically 
  to  unbounded patterns supported on  sets of measure zero
 (which resemble a sum of weighted Dirac measures).
  This is a new pattern formation phenomenon in the systems of reaction-diffusion-type which has not been studied analytically so far.

Additionally, we extend some results to nonlocal equations with general nonlinearities:
\begin{align}
u_t  &=   f(u,\xi),&     \text{for}\quad &x\in{\Omega}, \; t>0, && \label{eqs1}\\
\xi_t  &=   \int_\Omega g\big(u(x,t),\xi(t)\big)\;dx&  \text{for}\quad & t>0,&&\label{eqs2}
\end{align}
with arbitrary $C^1$-functions $f=f(u,\xi)$ and $g=g(u,\xi)$ and 
supplemented with suitable  initial conditions. Since the obtained results are only a slight modification of those shown recently for reaction-diffusion-ODE models, we place them in Appendix.

In Appendix \ref{sec:instab}, we show that, although model \eqref{eqs1}-\eqref{eqs2} is not a reaction-diffusion system, it may exhibit a pattern formation phenomenon based on the same principle as the classical Turing mechanism, {\it i.e.}~there exists a spatially constant steady state, which is stable to spatially homogeneous perturbations, but unstable to spatially heterogeneous perturbations. 
However, we prove much more than it us usually shown in the case of studies of DDI
phenomena in reaction-diffusion equations.
First,we characterize stationary solutions  of problem
\eqref{eqs1}-\eqref{eqs2}. Then, we show that they are unstable (more precisely, nonlinearly unstable in the Lyapunov sense) under so called {\it autocatalysis assumption}, {\it i.e.}~when $f_u>0$, which is a condition typical for models exhibiting Turing instability (see Theorem \ref{thm:instab} below for more details). 
This indicates that, in the considered class of models, {\it the Turing-type mechanism destabilizes not only constant steady states but also  non-constant stationary solutions.} A similar mechanism, where DDI destabilizes all non-constant stationary solutions of reaction-diffusion-ODE 
of the form 
\eqref{RD}, has been recently studied in our papers \cite{MKS12,MKS16}.

%

%

 Moreover, in Appendix \ref{app:blowup} we provide another particular example of system \eqref{eqs1}--\eqref{eqs2}, which shows  that the Turing phenomenon may lead not only to formation of unbounded patterns as those shown in Fig.~\ref{figure} but also to a blow up of solutions in a finite time.
Analogous results for reaction-diffusion-ODE systems \eqref{RD} with particular nonlinearities were published in Refs.~\cite{MCKSZ16, KSZ16}.

%


\section{Model of early carcinogenesis}\label{sec:carc}

\subsection{Preliminary properties of solutions}

For completeness of the paper, we provide basic properties of the nonlocal model  \eqref{eqsc1}-\eqref{inisc}.

\begin{prop}\label{thm:existence:carc}
Assume that   
 $u_0\in C (\overline\Omega)$ is nonnegative and  $\xi_0> 0$.
Then the initial value problem 
\eqref{eqsc1}-\eqref{inisc} has a unique, global-in-time, nonnegative solution 
$u \in C([0, \infty)), C (\overline\Omega))$, $\xi \in C^1([0, \infty))$.
This solution satisfies equation \eqref{eqsc1} in a classical sense, 
because  
$u(x,\cdot) \in C^1([0, \infty))$ for every $x\in \Omega$.
The following pointwise estimates  hold true:
\begin{equation}\label{est1:carc}
0\leq u(x,t)\leq e^{(a-d)t}u_0(x) \qquad \text{and}\qquad
0< \xi(t)\leq \max\left\{ \xi_0, {\kappa_0}\right\}
\end{equation}
for all $x\in \Omega$ and $t\geq 0$. Moreover, 
the ``total mass'' of $u(x,t)$ is bounded: 
\begin{equation}\label{mass:bound}
\sup_{t>0}\int_\Omega u(x,t)\,dx<\infty.
\end{equation}
\end{prop}

\begin{proof}
A construction of a unique local-in-time continuous solution 
to problem \eqref{eqsc1}-\eqref{inisc} on the set $\overline{\Omega} \times [0, T]$ with certain $T>0$ is more-or-less standard and 
we recall it  in the beginning of 
Appendix \ref{sec:instab}. 

This solution is nonnegative in case of nonnegative initial conditions, which can be proved in the following  way.  Let $T>0$ be arbitrary.
First, we notice that since $\sup_{x \in \overline{\Omega},\, t \in [0, T] }|u(x, t)| < \infty$, we have got
$
\sup_{0 \le t < T} \int_\Omega u^2 (x, t)\, dx < \infty. 
$
Suppose that there exists $T_1 \in (0, T]$ such that $\xi (T_1) = 0$
and $\xi(t)>0$ for $t\in [0, T_1]$.
It is easy to see by equation \eqref{eqsc2} that $\xi_t (T_1) = \kappa_0 > 0$ which implies immediately that $\xi (t)$ cannot 
decrease in a neighborhood of $T_1$.
 Hence, we obtain that $\xi (t) > 0$ for all $t \in [0, T]$. 
On the other hand, given an arbitrary $\xi(t)$, equation \eqref{eqsc1} is an ordinary differential equation for $u(x,\cdot)$ for each $x\in \overline\Omega$. This equation has a trivial solution $u\equiv 0$ for each $\xi(t)$. Hence, by a standard argument for ordinary differential equations involving the uniqueness of solution, we obtain that 
if for some $x\in \overline\Omega$ we have $u(x,0)=0$, then $u(x,t)=0$ for all $t\in [0,T]$ and the inequality $u(x,0)>0$ implies $u(x,t)>0$ for all $t\in [0,T]$.

%

Nonnegative local-in-time solutions  can be continued 
 global-in-time by a standard continuation argument
 provided we prove estimates \eqref{est1:carc} which may be obtained in the following way.
Applying to  equation \eqref{eqsc1}
 the inequality  $u\xi/(1+u\xi)\leq 1$ (valid for a nonnegative 
solution $(u,\xi)$)
we obtain the differential inequality $u_t\leq (a-d)u$ which implies first estimate in \eqref{est1:carc}.
The second one in \eqref{est1:carc} is a direct consequence of the inequality $\xi_t\leq-\xi+\kappa_0$ resulting from equation \eqref{eqsc2} for nonnegative $\xi$.

To show property \eqref{mass:bound}, we 
use  a differential inequality $u_t\leq a u^2\xi-du$ 
obtained  from equation \eqref{eqsc1} with
$u\xi\geq 0$. Integrating this inequality over $\Omega$ and using the equation for $\xi$ in \eqref{eqsc2}, we have got the estimate
\begin{equation}\label{diff:mass}
\begin{split}
\frac{d}{dt}\left(\int_\Omega u\,dx +a\xi\right) &\le -d\int_\Omega u\,dx -a\xi+a\kappa_0\\
&\leq -\min\{1,d\}\left(\int_\Omega u\,dx +a\xi\right)+a\kappa_0,
\end{split}
\end{equation}
which implies that the quantity $\int_\Omega u(t)\,dx +a\xi(t)$ is bounded for 
$t\in (0,\infty)$, because the 
constants $a$ and $d$ are positive. Since $\xi(t)>0$, we immediately obtain 
\eqref{mass:bound}.

Details of an analogous  proof in the case of a reaction-diffusion-ODE system corresponding to \eqref{eqsc1}-\eqref{eqsc2} can be found in \cite[Sec.~3]{MKS12}.
\end{proof}

Next, we discuss space homogeneous solutions of problem \eqref{eqsc1}-\eqref{inisc}.

\begin{prop}\label{thm:kinetic}
If $u_0(x)\equiv \bar u_0$ is independent of $x$,
then the corresponding solution of 
\eqref{eqsc1}-\eqref{inisc} 
is independent of $x$ as well. Thus, for $|\Omega|=1$,
the function
$u(x,t)=u(t)$ and $\xi=\xi(t)$
satisfy the following  system of ordinary differential equations
\begin{equation}\label{eq-kin}
\frac{d}{dt} u=\Big(\frac{a u\xi}{1+u\xi} -d\Big) u,\qquad 
\frac{d}{dt} \xi =  - \xi - \xi  u^2   +\kappa_0,
\end{equation}
which after supplementing with  initial data $\u_0>0$ and $\xi_0>0$,
has a unique global-in-time positive solution $(\u(t),\xi(t))$.
This solution is bounded for $t>0$.
\end{prop}
\begin{proof}
A solution of problem \eqref{eqsc1}-\eqref{inisc}  with constant $u_0(x)=\u_0$
does not depend on $x$ which is an immediate consequence 
of the uniqueness of solutions established in Proposition \ref{thm:existence:carc}.
From now, the study of the system of ordinary differential equations 
\eqref{eq-kin} is completely standard. In particular, we 
have  the
 differential inequality $du/dt \leq a u^2\xi-du$ for nonnegative solutions  
 which together with the second equation in \eqref{eq-kin} yields the estimate
$$
\frac{d}{dt}(u+a\xi) \le -du-a\xi+a\kappa_0
\leq -\min\{1,d\}(u+a\xi)+a\kappa_0.
$$
Thus, 
 the sum $u+a\xi$ 
(and so each of its term)
is bounded on $[0,\infty)$.
\end{proof}

\begin{rem}[Constant stationary solutions]\label{rem:DDI:carc}
System of ODEs \eqref{eq-kin} has a trivial steady state 
$(u,\xi)=(0,\kappa_0)$ which is its asymptotically stable solution
(see  also Remark \ref{rem:trivial}, below).
Detailed analysis of {\it positive} steady states of system \eqref{eq-kin}
can be found in our recent work \cite[Appendix B]{MKS16}. It is shown that for $a>d$ and $\kappa_0^2>4(d/(a-d))^2$, there exist two positive steady states  of system \eqref{eq-kin}.
One of these solutions is always unstable. 
Conditions 
on the coefficients in equations \eqref{eq-kin} under which the second constant solution is an asymptotically stable solution of the system
of ordinary differential equations \eqref{eq-kin} can be found in \cite[Appendix B]{MKS16}.
\end{rem}

\begin{rem}[Turing instability of constant solutions]
By Theorem \ref{cor:gs2} below, both constant steady states 
discussed in Remark \ref{rem:DDI:carc}
are unstable solutions to the nonlocal problem  
\eqref{eqsc1}-\eqref{inisc}.
In particular, we obtain that this problem describes 
  the Turing instability 
of this constant steady state which is stable as a solution
to the kinetic ODE system \eqref{eq-kin}.
\end{rem}


\subsection{Instability of spatially heterogeneous nonnegative stationary solutions}\label{sec:stat}
\ 
Now, we study non-constant {\it nonnegative} stationary solutions of system \eqref{eqsc1}-\eqref{eqsc2}, 
namely, we look for a function $U\in L^\infty(\Omega)$ and a constant $\bxi\in\R$ satisfying 
\begin{align}
\Big(\frac{a U\bxi}{1+U\bxi} -d\Big) U&=0 &&\text{for}\ x\in \Omega, \label{seqsc1}\\
 - \bxi - \bxi \int_\Omega U^2 \;dx  +\kappa_0&=0. && 
 \label{seqsc2}
\end{align}
For this purpose,
we decompose the set $\Omega$ into an arbitrary disjoint sum of two measurable sets 
$$
\Omega = \Omega_1\cup\Omega_2, \qquad \text{where}\quad \Omega_1\cap\Omega_2=\emptyset \quad\text{and}\quad |\Omega_1|>0,
$$
and solving  equation \eqref{seqsc1} with respect to $U$ we define
\begin{equation}\label{Ugs}
U(x)=
\begin{cases}
d/\left((a-d)\bxi \right)&\text{if}\quad  x\in \Omega_1,\\
0&\text{if}\quad  x\in \Omega_2.
\end{cases}
\end{equation}
Then, for $a>d$, one calculates $\bxi$ from equation \eqref{seqsc2}
which for $U$ defined by formula \eqref{Ugs} reduces to the quadratic equation
\begin{equation}\label{xigs}
\bxi^2 -\kappa_0 \bxi + \dfrac{d^2}{(a-d)^2} |\Omega_1|=0
\end{equation}
with two positive roots,  provided  $\kappa_0^2>4 \left(d/(a-d)\right)^2|\Omega_1|$.

Now, we prove  that all such stationary  solutions are unstable.

\begin{theorem}[Instability of all stationary solutions]\label{cor:gs2}
Nonnegative  stationary solution $\big(U,\bxi\big)$  of system \eqref{eqsc1}-\eqref{eqsc2} exist  under the assumptions $a>d$ and 
$\kappa_0^2>4 \left(d/(a-d)\right)^2|\Omega_1|$. They are 
 given by formula \eqref{Ugs} with $|\Omega_1|>0$ and $\bxi$ satisfying equation \eqref{xigs}. 
 
The couple $\big(U,\bxi\big)$ is an unstable solution of the initial value problem for the nonlocal  
system \eqref{eqsc1}-\eqref{eqsc2}.
\end{theorem}

\begin{proof}
The construction of nonnegative stationary solutions is given above. To show their instability, 
we apply Theorem \ref{thm:instab} from Appendix \ref{sec:instab}.
Here,  
the autocatalysis assumption \eqref{auto:constant} holds true 
because for $U(x) \bxi = d/(a-d)$ with $x\in \Omega_1$
  we have 
$$
f_u \big(U(x), \bxi\big) =
\frac{aU(x) \bxi}{1 + U(x)\bxi} -d + \dfrac{aU(x)\bxi}{\left(1 + U(x)\bxi\right)^2 }=
\dfrac{d(a-d)}{a} >0 \qquad \text{for all} \quad x\in \Omega_1.
$$
\end{proof}

Further discussion of stationary solutions of nonlocal systems with general 
nonlinearities is contained in Appendix \ref{sec:instab}.

\begin{rem}\label{rem:trivial}
On the other hand, one can prove that 
 the ``trivial'' stationary solution 
$\big(U(x),\bxi\big)=(0,\kappa_0)$ 
 is an asymptotically stable stationary solution  of system \eqref{eqsc1}-\eqref{eqsc2}.
Indeed, it follows from equation \eqref{eqsc1} that the nonnegative function $u(x,t)$ satisfies 
$u_t\leq au^2K-du$, where the constant $K=\sup_{t\geq 0} \xi(t)$ is finite by
second inequality in  \eqref{est1:carc}.
Hence, if $u_0\in L^\infty(\Omega)$ is nonnegative and sufficiently small 
then  $u(x,t)\to 0$ as $t\to\infty$ uniformly in $x \in \overline{\Omega}$. 
This decay of $u$ implies  $\int_\Omega u^2(x,t)\,dx\to 0$ as $t\to\infty$.
Hence, using equation \eqref{eqsc2}  we can easily show that $\xi(t)\to \kappa_0$ as $t\to\infty$.
See {\it e.g.} \cite[Theorem 2.2]{MKS12})  for an analogous   reasoning in the case of the corresponding 
reaction-diffusion-ODE system.
\end{rem}

\begin{rem}\label{rem:plateau}
Notice that if for some $x_1\neq x_2$ we have 
$u_0(x_1)=u_0( x_2) $, then
$u(x_1,t)=u(x_2,t)$ for all $t\geq 0$, because both quantities 
$u(x_1,t)$ and $u(x_2,t)$ as functions of $t$
satisfy equation \eqref{eqsc1} with the same function $\xi=\xi(t)$ and the same initial conditions. 
Consequently, if the measure of $\Omega_*\equiv \{x\in \Omega\;:\; u_0(x)=
\max_{x\in \overline\Omega} u_0(x)\}
$ is positive, 
the function $u(x,t)$ cannot escape to $+\infty$ for any $x\in \Omega$ due to the boundedness of the mass \eqref{mass:bound}.
In such a case,  
one can show (by using  Lemma \ref{lem:decay} below) that 
\begin{equation*}
u(x,t) \to
U(x)= 
\left\{
\begin{array}{ccl}
\u &\text{if}& x\in\Omega_*,\\
0&\text{if}&x\in\Omega\setminus \Omega_*
\end{array}
\right.
\qquad 
 \text{and}\qquad 
\xi(t)\to\bxi
\end{equation*}
as  $t\to\infty$, where $\big(U(x),\bxi\big)$ is a discussed-above 
stationary solution of system
 \eqref{eqsc1}-\eqref{eqsc2}.
Thus, there are bounded 
(and also continuous) initial conditions
such that the corresponding  solutions converge pointwise for every $x\in\Omega $
 towards discontinuous stationary solutions.
 Obviously,
such convergence result
 does not contradict  the instability 
of a steady state $(U,\bxi)$
proved in Theorem \ref{cor:gs2}.
We refer the reader to 
 Fig.~\ref{figure} (the graph in the second row and the second column)
  for a numerical illustration of such a phenomena.
\end{rem}


\section{Formation of unbounded spikes}
Now, we are in a position to prove our main result on unbounded 
growth of solutions to system \eqref{eqsc1}--\eqref{eqsc2}.

\begin{theorem}\label{thm:growth}
Let $\Omega\subset \R^n$ be a bounded and closed set with non-empty interior and satisfying 
$|\Omega|=1$.
Let $(u,\xi)$ be a nonnegative solution of problem \eqref{eqsc1}--\eqref{inisc},
where positive constants  in those equations 
satisfy the following inequalities
\begin{itemize}
\item the parameter $a$ is large:
\begin{align}
2(a-d)\geq 1, \label{as:a} 
\end{align}
\item the constant $\kappa_0$ is large:
\begin{align}
\kappa_0 \ge 4 a.
 \label{as:k0}
\end{align}
\end{itemize}
Let $\lambda$ satisfy
\begin{align}
\dfrac{1}{2} \le \lambda \le 1-\dfrac{2a}{\kappa_0}. \label{k-eq1}
\end{align}
Assume that 
 nonnegative initial conditions 
$ u_0 \in C (\Omega)\cap L^\infty(\Omega)$ and $\xi_0 \in \R$
 satisfy
\begin{equation}\label{la00}
   \xi_0 \int_\Omega u_0^2 (x)\, dx > \lambda \kappa_0
 \qquad
 \text{and}\qquad 
 0<\xi_0 \leq  (1-\lambda){\kappa_0}.
\end{equation}
Suppose  that the set 
$$
\Omega_{*}=\{x_*\in\Omega \,:\, u_0 (x_\ast) =
 \max_{x \in \Omega}u_0(x)\}
$$
has measure zero.
Then
\begin{align}
&\sup_{t>0}u(x_*,t)  =+\infty && \text{for each} \;\; x_*\in\Omega_*,\label{sup:u}\\
&\sup_{t>0}u(x,t) =\infty  && \text{for each}
 \;\; x\in\Omega\setminus \Omega_*,\\
&\inf_{t>0}\xi(t)  =0. &&\label{inf:xi}
\end{align}
\end{theorem}

%
%
%

We  show on Fig. \ref{figure} numerical simulations of solutions to 
the nonlocal  system  \eqref{eqsc1}-\eqref{eqsc2},  illustrating in this way 
the analytical results presented in Theorem~\ref{thm:growth}.

We proceed the proof  of Theorem \ref{thm:growth} by auxiliary results.


\begin{lemma}\label{lem:0}
Under the assumptions of Theorem  \ref{thm:growth},
the  solution $\big(u(x,t),\xi(t)\big)$ 
of problem \eqref{eqsc1}--\eqref{inisc} satisfies 
\begin{equation}\label{la0}
\xi(t) \int_\Omega u^2 (x, t)\, dx \ge \lambda \kappa_0 \qquad
 \text{and}\qquad 0 < \xi (t) \le (1-\lambda) \kappa_0
\end{equation}
for all $t \ge 0$. 
%

\end{lemma}
\begin{proof}
By the first inequality in \eqref{la00} and 
by the continuity 
of the solution $(u,\xi)$ 
({\it cf.}~Proposition \ref{thm:existence:carc}), there exists $T_1 >
 0$ such that 
\begin{equation}\label{la1}
 \int_\Omega u^2 (x, t)\, dx - \lambda \dfrac{\kappa_0}{\xi(t)} > 0 \qquad
 \text{for}\quad 0 \leq  t < T_1.
\end{equation}
Suppose that 
\begin{equation}\label{T1:0}
 \int_\Omega u^2 (x, T_1)\, dx - \lambda \dfrac{\kappa_0}{\xi(T_1)} = 0.
\end{equation}

First, notice that using \eqref{la1} in equation \eqref{eqsc2}
we obtain the differential inequality 
 $\xi_t \le -\xi + (1-\lambda)\kappa_0$. 
This implies $\xi (t) \le (1-\lambda)\kappa_0$ for $0 \le t \le T_1$
 due to the assumption \eqref{la00}. Thus, we obtain from
  \eqref{la1} and  \eqref{k-eq1}
the following estimate  
\begin{equation}\label{la2}
 \int_\Omega u^2 (x, t)\, dx \geq \lambda \dfrac{\kappa_0}{\xi (t)} \ge
  \dfrac{\lambda}{1-\lambda} \ge 1
\qquad
 \text{for}\quad 0 \leq  t \le T_1.
\end{equation}

Next, multiplying equation \eqref{eqsc1} by $u$ and integrating it over $\Omega$ results in the
equation
\begin{equation}\label{L2}
\frac12 \frac{d}{dt} \int_\Omega u^2(x,t)\;dx= 
 \int_\Omega \left(\frac{a u(x,t)\xi(t)}{1+u(x,t)\xi(t)} -d\right) u^2(x,t)\;dx.
\end{equation}
Hence, 
by a direct calculation involving  equations \eqref{L2} and \eqref{eqsc2}, we
 obtain the identity
\begin{equation}\label{L2bis}
\begin{split}
\dfrac{d}{dt}\left[\int_\Omega u^2 (x, t)\, dx - \lambda \dfrac{\kappa_0}{\xi
 (t)} \right] & = 2(a-d)\int_\Omega u^2 (x, t)\, dx - 2a 
 \int_\Omega \dfrac{u^2 (x, t)}{1 + u(x, t)\xi(t)}\, dx \\
&\quad - \lambda \dfrac{\kappa_0}{\xi (t)} - 
 \lambda \dfrac{\kappa_0}{\xi (t)} \int_\Omega u^2 (x, t)\, dx +
  \lambda \left(\dfrac{\kappa_0}{\xi (t)}\right)^2 .
\end{split}
\end{equation}
Here, using 
a minor rearrangement of terms on the right-hand side and 
the following simple inequalities
(which are valid because the solution is nonnegative and because $|\Omega|=1$)
$$
 \int_\Omega \dfrac{u^2 (x, t)}{1 + u(x, t)\xi(t)}\, dx
\leq \frac{1}{\xi(t)} \int_\Omega u (x, t)\, dx
\leq \frac{1}{\xi(t)} \left(\int_\Omega u^2 (x, t)\, dx\right)^{1/2},
$$
we obtain the lower bound
\begin{align*}
\dfrac{d}{dt}\left[\int_\Omega u^2 (x, t)\, dx -  \lambda \dfrac{\kappa_0}{\xi
 (t)} \right] 
\ge &- \left(\dfrac{\kappa_0}{\xi (t)} -1 \right) \left[\int_\Omega
 u^2 (x, t)\, dx -  \lambda \dfrac{\kappa_0}{\xi (t)}\right] \\
& + \big(2(a-d) -1 \big)
\int_\Omega u^2 (x, t)\, dx \\
&+ (1-\lambda) \dfrac{ \kappa_0}{\xi(t)}\int_\Omega u^2
 (x, t)\, dx  -  \dfrac{2a}{\xi (t)}\left(\int_\Omega u^2 (x, t)\,
 dx \right)^{1/2}.
\end{align*}
In this inequality, 
the second term on the right-hand side is nonnegative by  the assumption \eqref{as:a}. So is the difference of last two terms on the right-hand side,
because $\int_\Omega u^2 (x, t)\, dx\geq 1$ ({\it cf.} \eqref{la2})
and because 
$
(1-\lambda) \kappa_0 z-2az^{1/2}\geq 0
$
for all $z\geq 1$ if $\lambda \le 1-{2a}/{\kappa_0}$.

Therefore, for $0\leq t\leq T_1$, we have the differential inequality 
\[
 \dfrac{d}{dt}\left[\int_\Omega u^2 (x, t)\, dx - \lambda \dfrac{\kappa_0}{\xi
 (t)} \right] \ge - \left(\dfrac{\kappa_0}{\xi (t)} -1 \right) \left[\int_\Omega
 u^2 (x, t)\, dx - \lambda \dfrac{\kappa_0}{\xi (t)}\right],
\]
which for  $\xi (t)>0$ leads to  the estimate
\[
 \int_\Omega u^2 (x, t)\, dx - \lambda \dfrac{\kappa_0}{\xi (t)} \ge
 e^{-\left[\frac{\kappa_0}{\min_{0 \le t \le T_1} \xi (t)} - 1
 \right]t} \left[\int_\Omega u_0^2 (x)\, dx -
 \lambda \dfrac{\kappa_0}{\xi_0}\right] > 0.
\]
This inequality 
 for $t = T_1$ contradicts identity \eqref{T1:0}. %
Hence, we have completed the proof  of inequalities \eqref{la0} for all $t\geq 0$.

\end{proof}

Now, let us emphasize the following immediate consequence of
Lemma \ref{lem:0}. 

\begin{cor}\label{cor:max}
Let the assumptions of Lemma  \ref{lem:0} hold true.
Denote by $x_\ast \in \Omega$ a point of the maximum of $u_0$, namely, 
$
 u_0 (x_\ast) = \max_{x \in \Omega} u_0 (x).
$
Then
\begin{equation}\label{u*int}
 u^2 (x_\ast, t) \ge  \int_\Omega u^2(x,t)\,dx\geq 1
 \qquad \text{for all}\quad
 t > 0.
\end{equation}
\end{cor}
\begin{proof}
First inequality in \eqref{u*int} holds true because $|\Omega|=1$. 
The second one results immediately from inequalities \eqref{la0}
in the same way 
as in the proof of \eqref{la2}.
\end{proof}

\begin{lemma}\label{lem:decay}
Let the assumptions of Theorem  \ref{thm:growth} hold true.
Assume that $x_*\in\Omega_*$ and suppose 
that $u_*(t)\equiv u(x_*,t)=\max_{x\in \Omega} u(x,t)$ is a bounded function
of $t\in [0,\infty)$.
Then, for each $x\in\Omega$ such that $u_0(x)<u_0(x_*)$ it holds
$u(x,t)\to 0$ exponentially as $t\to\infty$.
\end{lemma}

\begin{proof}
If $u_*(t)$ is a bounded function, there exists a constant $R_1>0$ such that   $u_*(t)\leq R_1$ for all $t>0$. Thus,  we have
$\int_\Omega u^2(x,t)\,dx\leq R^2_1$ for all $t>0$ (because $|\Omega|=1$).
Hence, using equation \eqref{eqsc2} we obtain the differential inequality
$
\xi_t\geq -(1+R^2_1)\xi+\kappa_0,
$
which implies the lower bound 
\begin{equation}\label{xiR2}
\xi(t)\geq \min\left\{\xi_0,\frac{\kappa_0}{1+R^2_1}\right\}\equiv R_2 \qquad
\text{for all}\quad t\geq 0.
\end{equation}

Now, for simplicity of notation, we denote $u=u(x,t)$ and $u_*=u(x_*,t)$.
Hence, by a direct calculation involving equation \eqref{eqsc1},
we obtain
\begin{equation}\label{uu*}
\frac{\partial}{\partial t}\left(\frac{u}{u_*} \right)=-a \frac{u}{u_*}
\left(
\frac{u_*\xi (1-u/u_*)}{(1+u\xi)(1+u_*\xi)}
\right).
\end{equation}

This differential equation for the function $w=u/u_*$ implies the inequalities 
\begin{equation}\label{w<1}
w(t)=\frac{u(x,t)}{u(x_*,t)}\leq 
w(0)=\frac{u(x,0)}{u(x_*,0)}<1\quad  \text{for all} \quad t\geq 0.
\end{equation}
Moreover, using the estimate $u_*(t)\geq 1$ for all $t\geq 0$ 
({\it cf.} Corollary \ref{cor:max}), inequality \eqref{xiR2},
the bound $\xi(t)\leq \kappa_0$ ({\it cf.} \eqref{la0}),
 and
the estimate 
$$
(1+u\xi)(1+u_*\xi)\leq (1+R_1\kappa_0)^2
$$
we obtain the differential inequality
\begin{equation}\label{uu**}
\frac{\partial}{\partial t}\left(\frac{u}{u_*}\right) \leq
-a \frac{u}{u_*}
\left(
\frac{R_2 (1-u_0(x)/u_0(x_*))}{(1+R_1\kappa_0)^2}
\right)
\end{equation}
which implies the exponential decay in $t$ of $u/u_*$ because $u_0(x)/u_0(x_*)<1$.
However, since we assume that $u_*$ is bounded, we obtain immediately 
the exponential decay of 
 $u=u(x,t)$.
\end{proof}

We are in a position to complete the proof of our first theorem.

\begin{proof}[Proof of Theorem \ref{thm:growth}]

Suppose that $u=u(x,t)$ is bounded on $\Omega\times [0,\infty)$. Thus, by 
Lemma~\ref{lem:decay}, we have got 
$
u(x,t)\to 0
$
as $t\to\infty$
for every $x\in\Omega\setminus\Omega_*$. 
In particular, applying the Lebesgue dominated convergence theorem
we have $\int_\Omega u^2(x,t)\,dx\to 0$ as $t\to\infty$, because 
$|\Omega_*|=0$.
This is, however, in contradiction with the inequality in Corollary \ref{cor:max}. Hence, we conclude that $u_*(t)=\max_{x\in \Omega} u(x,t)$ is unbounded
for $t\in [0,\infty)$.

Assume  {\it a contrario} that $\sup_{t>0} u(x_1,t)=+\infty$ for some
$x_1\notin \Omega_{*}$. By the continuity of the initial datum $u_0$, the set
$\Omega_1\equiv \{x\in \Omega\,:\, u_0(x_1)<u_0(x)<u_0(x_*)\}$ has a positive
 Lebesgue measure.
Moreover,  using  differential equations for $w_1(x,t)=u(x_1,t)/u(x,t)$
and for $w_2(x,t)=u(x,t)/u(x_*,t)$, analogous to that one in \eqref{uu*}, we obtain
(in the same way as in the proof of inequalities \eqref{w<1}) that 
$$
u(x_1,t)<u(x,t)<u(x_*,t)\quad \text{for all $x\in \Omega_1$ and all $t\geq 0$}.
$$
These inequalities lead to a contradiction with  the boundedness of the mass 
\eqref{mass:bound}, because
$$
\sup_{t>0} \int_\Omega u(x,t)\,dx\geq 
\sup_{t>0} \int_{\Omega_1} u(x,t)\,dx\geq
\sup_{t>0} u(x_1,t)|\Omega_1|=+\infty.
$$
Thus, we have proved that 
\begin{equation}\label{u:fin}
\sup_{t>0} u(x,t)<\infty \qquad  \text{for each}\quad  x\in \Omega\setminus \Omega_*.
\end{equation}

Next, suppose   that there is a constant $\xi_1>0$ such that $\xi(t)\geq \xi_1$ for all $t>0$. Since we assume $a>d$ and since we have proved already that $\sup_{t>0} u(x_*,t)=\infty$,
we may find $t_1>0$ and $\delta >0$ such that 
\begin{equation}\label{in:delta}
\frac{au(x_*,t_1)\xi_1}{1+u(x_*,t_1)\xi_1}-d> \delta.
\end{equation} 
By the continuity of $u(x,t)$, inequality \eqref{in:delta} holds true at $t_1$ and 
  in a neighborhood $\mathcal{U}\subset \Omega$ of $x_*$,
such that $|\mathcal{U}|>0$.
 Moreover, using equation \eqref{eqsc1}
we immediately obtain the differential inequality
$$
u_t(x,t) \geq 
\left(\frac{au(x,t)\xi_1}{1+u(x,t)\xi_1}-d\right)u(x,t)
 \quad \text{for all $x\in \mathcal{U}$ and $t>0$},
$$
which 
together with inequality \eqref{in:delta} imply
$$
u_t(x,t) 
\geq \delta u(x,t) \quad \text{for all $x\in \mathcal{U}$ and $t\geq t_1$}.
$$
Hence, we have got the estimate $u(x,t)\geq e^{t\delta} u(x,t_1)$ for all $x\in \mathcal{U}$ and $t\geq t_1$, which contradicts
the boundedness of the mass
\eqref{mass:bound}, because the Lebesgue measure of $\mathcal{U}$ is greater than zero. This contradiction means that necessarily 
$ 
\inf_{t\geq 0} \xi(t)=0.
$ 
%
\end{proof}

\appendix

\section{Derivation of the nonlocal system}\label{ap:B}
In this part of this paper, 
 we show that solutions of the system   
\begin{align}
u_t  &=   f(u,\xi),&     \text{for}\quad &x\in{\Omega}, \; t>0, && \label{eqs1a}\\
\xi_t  &=   \int_\Omega g\big(u(x,t),\xi(t)\big)\;dx&  \text{for}\quad & t>0,&&\label{eqs2a}
\end{align}
with arbitrary $C^1$-functions $f=f(u,\xi)$ and $g=g(u,\xi)$ and 
supplemented with the initial conditions
\begin{eqnarray}\label{inis}
&&u(\cdot,0)  =   u_{0}\in L^\infty(\Omega),\qquad \xi(0)  =  \xi_{0}\in\R,
\end{eqnarray}
 are limits of solutions to reaction-diffusion-ODE systems \eqref{RD}
when $D\to\infty.$
First, however, we recall certain properties of the heat semigroup.

\begin{lemma}\label{lemma:A}
Let $\{e^{tD\Delta}\}_{t\geq 0}$ be the heat semigroup with the Neumann boundary condition
in  a bounded domain $\Omega\subset \R^n$ with a smooth 
boundary and such that $|\Omega|=1$.
\begin{itemize}
\item[i.]
For every constant $C\in \R$, we have $e^{tD\Delta} C=C$ for all $t\geq 0$.

\item[ii.] 
For every $w_0\in  L^1(\Omega)$ there exists a number $ C(\|w_0\|_1) >0$ independent of $D>0$ such that
we have
\begin{equation}\label{lim:A}
 \sup_{t>0} 
\left(t^{n/2}\left\|e^{tD\Delta} \left( w_0-\int_\Omega w_0\,dx\right)\right\|_\infty\right)
\leq C(\|w_0\|_1) D^{-n/2}
\end{equation}
for all $D>0$.
\end{itemize}
\end{lemma}

\begin{proof}
The first part of this lemma is well-known 
 because every constant $C\in\R$ is a solution of the heat equation 
 in a bounded domain 
 with the Neumann boundary conditions.

To show the second part, we recall the following  estimate 
(see {\it e.g.} \cite[p.~25]{R84} and \cite[Prop.~12.5]{A83})
 of the heat semigroup with the Neumann boundary conditions: 	each $1\leq q\leq p\leq \infty$  and  every $z_0\in L^q(\Omega)$ such that
$\int_\Omega z_0\,dx=0$, we have
\begin{equation}\label{heat}
\begin{split}
\|e^{tD\Delta}z_0\|_p&\leq C\left(1+(tD)^{-(n/2)(1/q-1/p)}\right) e^{-\lambda_1 D t} \|z_0\|_q
\end{split}
\end{equation}
for all  $t>0$,
where $\lambda_1>0$ denotes the first nonzero
eigenvalue of  $-\Delta$ in $\Omega$  under the Neumann boundary conditions
and the number $C>0$ is independent of $t$, $D$, and $z_0$.
We use   inequality \eqref{heat} 
 with  $z_0=w_0-\int_\Omega w_0\,dx$  in the following way:
\begin{equation}\label{lem:A}
\left\|e^{tD\Delta} \left( w_0-\int_\Omega w_0\,dx\right)\right\|_\infty \leq 
C\left(1+(tD)^{-n/2}\right) e^{-\lambda_1 D t}
\left\| w_0-\int_\Omega w_0\,dx\right\|_1
\end{equation}
for all   $t>0$ and a constant $C>0$ 
independent of $t$, $D$, and $w_0$.
Since
$$
\sup_{t>0}t^{n/2}\left(1+(tD)^{-n/2}\right) e^{-\lambda_1 D t}
= D^{-n/2} 
\sup_{s>0} s^{n/2} \left(1+s^{-{n/2}}\right) e^{-\lambda_1  s} <\infty,
$$
we obtain immediately estimate \eqref{lim:A}.
%
\end{proof}

\begin{theorem}\label{thm:A}
Let $f$ and $g$ be arbitrary  $C^1$-nonlinearities and 
let $\Omega\subset \R^n$ be a bounded domain with a smooth boundary and such that $|\Omega|=1$.
Fix arbitrary  $u_0,v_0\in L^\infty(\Omega) $ and $T>0$.
Assume that, for each $D>0$,  the couples $(u^D,v^D)$ are solutions of the initial-boundary 
value problem 
\begin{align}
\label{eq:A}&u^D_t=f(u^D,v^D),&& v_t^D=D\Delta v^D +g(u^D,v^D)&&\text{in}\quad \Omega\\
&u^D(x,0)=u_0(x),&&v^D(x,0)=v_0(x)&&\text{in}\quad \Omega\\
\label{ini:A}&&&\partial_n v^D=0&&\text{on}\quad \partial \Omega
\end{align}
on the common time interval $[0,T]$.
Suppose that
\begin{equation}\label{as:A}
 \sup_{D>0} \left(
\sup_{0\leq t\leq T} \|u^D(t)\|_\infty
+
\sup_{0\leq t\leq T} \|v^D(t)\|_\infty
\right) <\infty.
\end{equation}
Then
\begin{equation*}
\lim_{D\to\infty}  \sup_{0\leq t\leq T} t^{n/2}\left(\|u^D(t)-u(t)\|_\infty
+
 \|v^D(t)-\xi(t)\|_\infty\right)=0,
\end{equation*}
where $(u,\xi)$ is a solution of the nonlocal  system  \eqref{eqs1a}-\eqref{eqs2a}
with $\xi_0=\int_\Omega v_0(x)\,dx$.
\end{theorem}

\begin{proof}
A solution of the initial-boundary value problem \eqref{eq:A}-\eqref{ini:A} satisfies the system of the integral equations
\begin{align*}
&u^D(t) = u_0+\int_0^t f\big(u^D(s),v^D(s)\big)\,ds\\
&v^D(t) = e^{tD\Delta} v_0+\int_0^t  e^{(t-s)D\Delta} g\big(u^D(s),v^D(s)\big)\,ds.
\end{align*}
Subtracting from these equations an analogous  integral representation of $(u,\xi)$ (see equations  \eqref{int1}-\eqref{int2} below)
and calculating the $L^\infty$-norm we obtain the inequalities
\begin{eqnarray*}
\|u^D(t)-u(t)\|_\infty &\leq &\int_0^t \big\| f(u^D(s),v^D(s))-f(u(s),\xi(s))\big\|_\infty\,ds\\
\|v^D(t)-\xi(t)\|_\infty &\leq  &\left\|e^{tD\Delta} \left(v_0-\int_\Omega v_0\,dx\right)\right\|_\infty\\
&&+\int_0^t \left\| e^{(t-s)D\Delta} \left(g(u^D(s),v^D(s))-\int_\Omega g(u^D(s),v^D(s))\,dx \right)\right\|_\infty\,ds,\\
&&+ \int_0^t \int_\Omega \big| g(u^D(s),v^D(s))-g(u(s),\xi(s))\big|\,dx\,ds.
\end{eqnarray*}
Here, we have applied  part i.~of Lemma \ref{lemma:A}. 

Using the Taylor expansion  and 
assumption \eqref{as:A}, we find a constant $C$ independent of $D$ and of $s$
such that  
$$
 \big\| f\big(u^D(s),v^D(s)\big)-f(u(s),\xi(s))\big\|_\infty
\leq C\left( \|u^D(s)-u(s)\|_\infty +\|v^D(s) - \xi(s)\|_\infty\right)
$$
and 
$$
 \big\| g\big(u^D(s),v^D(s)\big)-g(u(s),\xi(s))\big\|_\infty
\leq C\left( \|u^D(s)-u(s)\|_\infty +\|v^D(s) - \xi(s)\|_\infty\right).
$$
Consequently,  using  all these inequalities
 together  with  the Gronwall lemma and multiplying by $t^{n/2}$, we obtain
\begin{equation}\label{last:A}
\begin{split}
t^{n/2} \Big( \|&u^D(t)-u(t)\|_\infty+\|v^D(t)-\xi(t)\|_\infty\Big)\leq 
\Bigg(t^{n/2}\left\|e^{tD\Delta} \left(v_0-\int_\Omega v_0\,dx\right)\right\|_\infty\\
&+t^{n/2}\int_0^t \left\| e^{(t-s)D\Delta} \left(g(u^D(s),v^D(s))-\int_\Omega g(u^D(s),v^D(s))\,dx \right)\right\|_\infty\,ds
\Bigg) e^{CT}.
\end{split}
\end{equation}

The first term on the right-hand side of inequality \eqref{last:A} tends to zero uniformly in $t\geq 0$
as $D\to\infty$ due to Lemma \ref{lemma:A}.ii.
To deal with the second term, we apply the heat semigroup estimate  \eqref{heat} with $p=\infty$ and with  fixed $q>n/2$
to obtain 
\begin{equation*}
\begin{split}
t^{n/2}\int_0^t &\left\| e^{(t-s)D\Delta} \left(g(u^D(s),v^D(s))-\int_\Omega g(u^D(s),v^D(s))\,dx \right)\right\|_\infty\,ds\\
&\leq t^{n/2}\int_0^t \big(1+(t-s)D)^{-(n/2)(1/q)}\big)e^{-(t-s)D\lambda_1}
R(D,q, s)\,ds,
\end{split}
\end{equation*}
where
$$
R(D,q, s)\equiv 
\left\| \left(g(u^D(s),v^D(s))-\int_\Omega g(u^D(s),v^D(s))\,dx \right)\right\|_q.
$$
Here, $\sup_{D>0, \,0\leq s\leq T} R(D,q,s)<\infty$ 
for each $q\in [1,\infty]$ 
by
assumption \eqref{as:A}. 
Moreover, since $n/(2q)<1$, we can easily find a finite number
 $C=C(T,q)>0$  by a change of variables such that 
$$
\sup_{0\leq t\leq T} t^{n/2} \int_0^t \big(1+(t-s)D)^{-(n/2)(1/q)}\big)e^{-(t-s)D\lambda_1}\,ds \leq C D^{-1}.
$$
Thus, computing ``$\sup_{0\leq t\leq T}$'' on both sides of inequality 
\eqref{last:A} we complete the proof of Theorem \ref{thm:A}. 
\end{proof}

\begin{rem}[Shadow-type limit for the  model of early carcinogenesis.]
\label{rem:shadow:carc}
It is rather standard to show that  the following 
initial-Neumann boundary value problem for the 
reaction-diffusion-ODE system 
\begin{align}
&u_t=\Big(\frac{a uv}{1+uv} -d\Big) u &\text{for}& \quad x\in \Omega, \; t>0, \nonumber\\
&v_t =  D\Delta v  - v - v u^2  +\kappa_0 &
\text{for}&\quad  x\in \Omega,\; t>0, \nonumber\\
&\partial_n v(x,t)=0 &\text{for}& \quad x\in \partial \Omega, \; t>0, \label{rd:carc}\\
&u(x,0)=u_0(x)\geq 0&\text{for}& \quad x\in  \Omega, \nonumber\\
&v(x,0)=\v_0(x)\geq 0 &\text{for}& \quad x\in  \Omega\nonumber 
\end{align}
has a unique global-in-time solution for each $D>0$ and every initial condition
$u_0,v_0\in L^\infty (\Omega)$. For the proof of this claim, it suffices to follow the reasoning from papers \cite[Sec.~3]{MKS12} and  \cite{MKS16}, where such models have been discussed in detail.
Now, for a fixed initial datum $u_0,v_0\in L^\infty (\Omega)$, we are going to show that the family of solutions $(u,v)=(u^D,v^D)$ of problem \eqref{rd:carc}
is uniformly bounded with respect to $D>0$ on each finite time interval $[0,T]$,
 as required in condition \eqref{as:A}.
Indeed, applying to  first equation in \eqref{rd:carc}
 the inequality  $uv/(1+uv)\leq 1$ (valid for every  nonnegative 
solution $(u,v)$)
we obtain the differential inequality $u_t\leq (a-d)u$ which implies the estimate
\begin{equation}\label{u:rem}
0\leq u(x,t)\leq u_0(x)e^{(a-d)t} \qquad \text{for all} \quad x\in\Omega, \; t>0.
\end{equation} 
Next, by a comparison principle 
for parabolic equations, we obtain that $v(x,t)$ can be estimated from above 
by a solution $\v=\v(t)$ of the Cauchy problem 
$$
\frac{d}{dt} \v = -\v+\kappa_0, \quad \v(0)=\|v_0\|_\infty.
$$ 
Hence,
\begin{equation}\label{v:rem}
0\leq v(x,t)\leq \v(t)\leq \max \{\kappa_0, \|v_0\|_\infty\}
 \qquad \text{for all} \quad x\in\Omega, \; t>0.
\end{equation}
Both estimates \eqref{u:rem} and \eqref{v:rem} imply that the assumption 
\eqref{as:A} hold true. Hence, by Theorem \ref{thm:A}, solutions of the initial 
boundary value problem \eqref{rd:carc} converge as $D\to\infty$ toward the solution of problem \eqref{eqsc1}-\eqref{inisc} with $\xi_0=\int_\Omega v_0(x)\;dx$.
\end{rem}

\begin{rem}
It is well-known  that for a system of two reaction-diffusion equations
\begin{equation}\label{RD2}
u_t  = \varepsilon \Delta u+  f(u,v), \qquad v_t  =   D \Delta v+g(u,v),
\end{equation}
 with $\varepsilon>0$ and $D>0$, a regular perturbation problem is
obtained, under some conditions,  by passing to the limit $D\to\infty$.
The obtained system of a reaction-diffusion equation coupled to an
ordinary differential equation with a nonlocal term 
(as the one in \eqref{eqs2})
is exhibiting
dynamics qualitatively similar to that of  the original reaction-diffusion
system with the diffusion coefficient $D$ being large. It is called a {\it shadow
system} and it is an example of a model with nonlocal
kinetics. 
Shadow systems  have been introduced  by Keener  \cite{Keener78} and their 
properties have been established {\it e.g.}~in Ref.~\cite{Keener78, Nishiura82,
HaleSakamoto89, NPY01}. Analysis of shadow systems has provided
insights into dynamics of the activator-inhibitor model and of other
reaction-diffusion models under certain conditions
\cite{HaleSakamoto89}. The necessity of the conditions given 
in Ref.~\cite{HaleSakamoto89} is
highlighted by showing discrepancies between the dynamics of a shadow system
and the corresponding reaction-diffusion system in \cite{LiNi2009},
{\it i.e.} blow-up in finite time versus global existence.
Let us emphasize that, in this work, we consider the shadow approximation of system \eqref{RD2} with $\varepsilon =0$. Such systems give a singular limit of  reaction-diffusion models with small $\varepsilon>0$. Moreover, since they arise in modeling of processes with non-diffusing components, as described above, it is important to understand how their dynamics differ from dynamics of non-degenerated systems.
\end{rem}

\section{Instability of stationary solutions}\label{sec:instab}

A study of the general nonlocal initial value problem for the system 
\eqref{eqs1}-\eqref{eqs2}, namely, 
\begin{align}
&u_t  =   f(u,\xi),&     \text{for}\quad &x\in{\Omega}, \; t>0, && \label{Beq1}\\
&\xi_t  =   \int_\Omega g\big(u(x,t),\xi(t)\big)\;dx&  \text{for}\quad & t>0,&&\label{Beq2}\\
&u(x, 0)=u_0(x), \quad  \xi(0)=\xi_0 & & & & \label{Bini}
\end{align}
should begin by noticing 
that it has a unique local-in-time solution for every $u_0\in L^\infty(\Omega)$, $\xi_0\in\R$, and for arbitrary locally Lipschitz nonlinearities $f=f(u,\xi)$ and $g=g(u,\xi)$.  For the proof of this claim, it suffices to apply the Banach 
fixed point theorem to the following integral formulation of problem 
 \eqref{Beq1}-\eqref{Bini}
\begin{align}
u(x,t)&=u_0(x)+\int_0^t f\big(u(x,s),\xi(s)\big)\,ds,\label{int1}\\
\xi(t)&=\xi_0+\int_0^t\int_\Omega g\big(u(x,s),\xi(s)\big)\;dx\,ds\label{int2}
\end{align}
in order to obtain a solution $u\in C([0,T],L^\infty(\Omega))$ and
$\xi\in C([0,T])$ for some $T>0$ depending on  initial conditions and on nonlinearities. Then, a classical argument applied to system \eqref{int1}-\eqref{int2} allows us to show that, in fact,
$u(x,\cdot), \xi(\cdot)\in C^1([0,T])$ for every $x\in \Omega$.
Moreover,  if $u_0\in C(\Omega)$, then $u\in C([0,T]\times \Omega)$
(see {\it e.g.} \cite[Ch.~3]{Z86} for 
  results on differential equations in Banach spaces).

Our goal in this part of Appendix is to study stability properties of stationary solutions 
of the general nonlocal  system \eqref{Beq1}-\eqref{Beq2}. Here, a couple 
$(U,\bxi)\in L^\infty(\Omega)\times \R$ is called a stationary solution if
\begin{align}
f\big(U(x),\bxi\big)&=0    \qquad  \text{almost everywhere in $\Omega$,}\label{seq1}\\
\int_\Omega g\big(U(x),\bxi\big)\;dx&=0. \label{seq2}
\end{align}
 Now, if equation \eqref{seq1} can be solved (locally and not necessarily uniquely) 
with respect to $U(x)$, we obtain that $U$ has to be constant on a subset of $\Omega$. This is indeed  the case of the particular model
of early carcinogenezis 
 discussed by us in Section \ref{sec:stat}, where a characterization of all stationary solutions is possible.

Our main result on stationary solutions to the nonlocal  system \eqref{Beq1}-\eqref{Beq2} provides a~simple and natural condition under which a steady state is unstable.

\begin{theorem}[Instability of stationary solutions]\label{thm:instab}
Let $f=f(u,\xi)$ and $g=g(u,\xi)$ be arbitrary $C^2$-functions.
Assume that there exists $\Omega_1\subset \Omega$  with
$|\Omega_1|>0$, a constant $\u\in\R$, and  
 a solution $(U,\bxi)$
of  system \eqref{seq1}-\eqref{seq2} such that $U(x)=\u$ 
 for all $x\in\Omega_1$.
 If the autocatalysis condition holds, i.e. if
\begin{equation}\label{auto:constant}
f_u (\u, \bxi) > 0,
\end{equation}
 then $(U,\bxi)$ is  an unstable solution (in the Lyapunov sense) 
 of
 the nonlocal problem 
\eqref{Beq1}--\eqref{Bini}.
\end{theorem}

\begin{proof}
 In conformity with regular practice, we consider an initial value problem for the perturbation
$w(x,t)=u(x,t)-U(x)$ and  $\eta(t)=\xi(t) -\bxi,$
where $(u,\xi)$ is a solution of the nonlocal  problem \eqref{Beq1}-\eqref{Bini}
and $(U,\bxi)$ is a stationary solution satisfying the assumptions of 
Theorem~\ref{thm:instab}.
Thus, the couple $z=(w,\eta)$ is a solution of the initial value problem
\begin{equation}\label{perturb}
z_t=\L z+\mathcal{N}(z),\qquad z(0)=z_0\equiv(u_0-U, \xi_0-\bxi), 
\end{equation}
where 
\begin{equation}\label{L}
\L z=\L
\left(
\begin{array}{c}
w(x)\\
\eta
\end{array}
\right)
\equiv
\left(
\begin{array}{c}
f_u\big(U(x),\bxi\big)w(x)+f_\xi\big(U(x),\bxi\big)  \eta\\
\int_\Omega g_u\big(U(x),\bxi\big)w(x)\,dx+ \int_\Omega g_\xi \big(U(x),\bxi\big) \eta\,dx
\end{array}
\right)
\end{equation}
and $\mathcal{N}$ is a nonlinear term obtained in a usual way via the Taylor expansion 
from the nonlinearities in system \eqref{Beq1}-\eqref{Beq2}.

The linear operator $\L: L^\infty(\Omega)\times \R \to L^\infty(\Omega)\times \R$ is bounded,
hence, it generates a strongly continuous semigroup (in fact, a group) 
of linear operators on the Banach space
$X=L^\infty(\Omega)\times \R$ equipped with the usual norm
$\|(w,\eta)\|_X\equiv \|w\|_{L^\infty(\Omega)}+|\eta|$.

Now, we show that the number $\lambda_0=f_u(\u,\bxi)>0$ ({\it cf.} assumption \eqref{auto:constant}) is an eigenvalue of $\L$. 
To do it, it suffices to check that  
$
\bar z =(w_0, 0)
$
is the corresponding eigenvector for every non-trivial  $w_0\in L^\infty(\Omega)$
satisfying  $\int_{\Omega_1} w_0(x)\,dx=0$ and $w_0(x)=0$ for all $x\in \Omega\setminus \Omega_1$. One can always construct such bounded, non-trivial function $w_0$ due to the condition
$|\Omega_1|>0$.  Thus, by the assumptions on $U(x)$, we have
$$
\int_\Omega g_u\big(U(x),\bxi\big)w(x)\,dx=
g_u\big(\u,\bxi\big)\int_{\Omega_1} w(x)\,dx=0
$$
and, consequently, we obtain
$\L(w_0,0)^{T}=\lambda_0(w_0,0)^{T}$.

Finally, since $U$ is a bounded function,
using the Taylor expansion 
of the $C^2$-functions $f=f(u,\xi)$ and $g=g(u,\xi)$
we find  two constants $R>0$ and $C>0$ such that the the nonlinear term $\mathcal{N}$ in
\eqref{perturb} satisfies
$$
\|\mathcal{N}(z)\|_{X} \leq C\|z\|_{X}^2\qquad \text{for all}\quad \|z\|_X\leq R.
$$ 

We have thus checked all assumptions of \cite[Theorem 1]{SS00} which assure that the zero solution
of the initial value problem \eqref{perturb} is nonlinearly unstable
in the Laypunov sense. 
\end{proof}




\section{Blowup of solutions in finite time}\label{app:blowup}

A nonlocal effect caused by the integral over $\Omega$ in system \eqref{eqs1}-\eqref{eqs2} may lead not only to the instability of   steady states, but also to a blow-up of space-heterogeneous solutions, even in the case when space homogeneous solutions are global-in-time and uniformly bounded on the time half-line 
$[0,\infty)$.
In this part of Appendix, we describe  this phenomenon 
in the  case of a particular nonlocal  problem with 
a well-known nonlinearity from 
 mathematical
biology. 
More precisely, we consider a nonlocal  problem with  the nonlinearity
 as in the celebrated
 Gray-Scott system describing pattern formation in chemical reactions ~\cite{GrayScott}:
\begin{align}
\label{eq1gs} u_t  &=  - (B+k) u + u^2 \xi&     \text{for}\quad &x\in{\Omega}, \; t>0, && \\
\label{eq2gs} \xi_t  &=   -\xi \int_\Omega u^2 \,dx+B(1-\xi)&  \text{for}\quad & t>0,&&
\end{align}
where  $B$ and $k$ are positive constants.
As before we assume $|\Omega|=1$, hence, this is a particular case of system 
\eqref{eqs1}-\eqref{eqs2} with
$f(u, \xi)=- (B+k) u + u^2 \xi$ and 
  $g(u,\xi)=-\xi u^2 +B(1-\xi)$. 

Let us first formulate preliminary properties of solutions to the initial value problem for system
\eqref{eq1gs}-\eqref{eq2gs}.

\begin{prop}\label{prop:GS}
System \eqref{eq1gs}-\eqref{eq2gs} supplemented with an initial condition
$(u_0,\xi_0)\in L^\infty(\Omega)\times \R$ has a unique solution on an interval
$[0,T_{max})$ with certain maximal  $T_{max}\in (0,\infty]$. 
If $u_0\geq 0$ almost everywhere in $\Omega$ and $\xi_0\geq 0$, 
then $u(x,t)\geq 0$ almost everywhere in $\Omega$ and $\xi(t)\geq 0$ for all
$t\in [0, T_{max})$. 
For every nonnegative $\xi_0$, we have the estimate
\begin{equation}\label{xi:est}
0\leq \xi(t)\leq \max \{\xi_0, 1\}\qquad \text{for all} \quad t\in [0, T_{max}).
\end{equation}
\end{prop}

We skip the proof of this proposition because it is completely 
analogous to the proof of Proposition \ref{thm:existence:carc}.
Here, let us only mention that the upper bound  \eqref{xi:est} 
is an immediate consequence of the differential inequality 
$\xi_t\leq B(1-\xi)$ which is obtained from \eqref{eq2gs} with nonnegative $\xi(t)$.

\begin{rem}\label{rem:carc:stab}
We skip the discussion of stability properties of  stationary solutions 
the nonlocal  system  
\eqref{eq1gs}-\eqref{eq2gs}, because it is completely analogous 
to that one 
in Appendix \ref{sec:instab}, 
in the case of model \eqref{eqsc1}-\eqref{eqsc2}.
Here, let us only mention that piecewise constant 
stationary solutions exist and they are all unstable 
(because an autocatalysis condition is satisfied)
except  
the trivial  steady state  $(U,\bxi)=(0,1)$.
\end{rem}

Since all nontrivial stationary solutions are unstable, a question arises as to what is the long-time behavior of (large) solutions to the initial value problem for system \eqref{eq1gs}-\eqref{eq2gs}.
First, we emphasize in the following corollary   that space homogeneous nonnegative solutions 
({\it i.e} when $u$ does not depend on $x$)
are global-in-time and bounded. 
Here, we recall that such solutions satisfy the corresponding 
system of ordinary differential 
equations  under
our standing assumption $|\Omega|=1$.

\begin{prop}\label{prop:k:sss}
All  solutions $(u,\xi)=\big(u(t),\xi(t)\big)$ of the following initial value problem
for ordinary differential equations
\begin{align}
&\frac{d}{dt}u  =  -(B+k)u+u^2\xi,  \qquad \frac{d}{dt}\xi  =  - \xi u^2+B(1-\xi)\label{ksss}\\
&u(0)  =   u_{0}\geq 0,\qquad \xi(0)  =  \xi_{0}\geq 0 \label{kinisss}
\end{align}
are nonnegative, global-in-time, and uniformly bounded for $t>0$. 
\end{prop} 

\begin{proof}
The proof of this proposition  is completely standard if we observe that all nonnegative
solutions of problem \eqref{ksss}-\eqref{kinisss} satisfy the relation
$$
\frac{d}{dt} \big( u(t) + \xi(t)   \big)
=-(B+k)u(t)-B\xi(t)+B\leq 
-B\big(u(t)+\xi(t)\big)+B.
$$ 
Hence, as long as  $u$ and $\xi$ stay nonnegative,  the sum 
$u(t)+\xi(t)$ has  to be  bounded 
on the half line $[0,\infty)$.
\end{proof}


Our main result on system \eqref{eq1gs}-\eqref{eq2gs} 
ascertains
 that 
 a space inhomogeneity of  initial data may leads not only to instability but also to 
a   blow-up in finite time of the corresponding solution.

\begin{theorem}\label{thm:sss}
Fix $x_0\in \Omega$ and assume  that $u_0\in L^\infty (\Omega)$ satisfies
\begin{equation}\label{as:u0:1}
 0\leq u_0(x)<u_0(x_0)\qquad \text{for all}\quad 
x\neq x_0
\end{equation}
and 
\begin{equation}\label{as:u0:2}
A_0\equiv \int_\Omega \left(\frac{u_0(x_0)u_0(x)}{u_0(x_0)-u_0(x)}\right)^2\;dx<\infty.
\end{equation}
Assume also that
\begin{equation}\label{as:xi:blow}
\frac{1}{B+k}\min \left\{\xi_0, \frac{B}{A_0+b}\right\}>\frac{1}{u_0(x_0)}.
\end{equation}
Then, the corresponding solution of system 
\eqref{eq1gs}-\eqref{eq2gs} supplemented with the initial conditions $u(x,0)=u_0(x)$
and $\xi(0)=\xi_0$
blows up in a finite time at $x_0$  in the following sense.
There exists $T_{max}\in (0,\infty)$ such that 
\begin{itemize}
\item the solution $\big(u(x,t), \xi(t)\big)$ exists on $\Omega\times [0,T_{max})$
and it is continuous on $[0,T_{max})$ for every fixed $x\in\Omega$;
\item $u(x_0,t)$ blows up at $T_{max}$:
$
u(x_0,t)\to +\infty \quad \text{when}\quad t\to T_{max},
$
\item
 the following estimates hold true for all 
$ (x,t)\in (\Omega\setminus \{x_0\})\times [0,T_{max})$:
\begin{equation}\label{u:est:blow}
0\leq u(x,t)\leq \frac{u_0(x_0)u_0(x)e^{-t(B+k)}}{u_0(x_0)-u_0(x)} 
\end{equation}
and
\begin{equation}\label{xi:est:blow}
  \min \left\{\xi_0, \frac{B}{A_0+B}\right\}   \leq \xi(t)\leq \max \left\{\xi_0, 1\right\}.
\end{equation}
\end{itemize}
\end{theorem}

Notice that, for an initial condition  described in Theorem \ref{thm:sss}, the 
corresponding  $u(x,t)$ escapes to $+\infty$ for $x=x_0$ as $t\to T_{max}$ and remains 
bounded for all other $x\in \Omega$. On the other hand, the function $\xi(t)$ is bounded and separated from zero on the interval
$[0, T_{max})$.

\begin{rem}\label{rem:sharp}
The number $A_0$ defined in \eqref{as:u0:2} is finite if, for example, 
there exist constants $C>0$ and $\ell \in (0, n/2)$ such that  
$u_0(x_0)-u_0(x)\geq C|x_0-x|^{\ell}$
(or equivalently, 
$u_0(x)\leq u_0(x_0)-C|x_0-x|^\ell$)
in a neighborhood of $x_0$.
If $u_0$ is a $C^2$-function and strictly concave, then we have $u_0(x_0)-u_0(x)\leq C|x_0-x|^{2}$
in a neighborhood of $x_0$, because $u_0$ has a global maximum at $x_0$. 
Thus, the constant $A_0$  in \eqref{as:u0:2} is finite in dimension $n\leq 4$ if
$u_0$ is  more ``sharp'' at the maximum point $x_0$ than  a 
$C^2$-function. 
However, our numerical simulations 
performed for a model of early carcinogenesis considered in Section~\ref{sec:carc}
suggest that 
such assumptions 
may not be optimal 
and an unbounded growth of spikes could be  possible for smooth initial 
conditions, as well. 
\end{rem}

\begin{proof}[Proof of Theorem \ref{thm:sss}]
By Proposition \ref{prop:GS},
  the solution $(u,\xi)$ of the initial value problem for system \eqref{eq1gs}-\eqref{eq2gs} exists on a maximal time interval $[0,T_{max})$ and it is nonnegative. Moreover, the function $\xi(t)$ satisfies the upper bound in \eqref{xi:est:blow} which is an immediate consequence of Proposition \ref{prop:GS}.

For fixed $\xi$ and for each $x\in \Omega$,
we solve equation \eqref{eq1gs} proceeding in the usual way: 
first, one should check that $w(x,t)=u(x,t)e^{t(B+k)}$ satisfies the
ordinary differential 
 equation $w_t=w^2\xi e^{-t(B+k)}$ with separate variables. Thus, the function $u$ 
 can be expressed via $\xi$ in the following way
\begin{equation}\label{u:form}
u(x,t)=\dfrac{e^{-t(B+k)}}{\dfrac1{u_0(x)}-\int_0^t \xi(s)e^{-s(B+k)}\;ds}.
\end{equation}
By assumption \eqref{as:u0:1}, we have 
$1/u_0(x)> 1/u_0(x_0)$ for all $x\in\Omega \setminus \{x_0\}$;
thus, it follows form formula \eqref{u:form}
that the solution $\big(u(x,t), \xi(t)\big)$  of  \eqref{eq1gs}-\eqref{eq2gs}
exists for all $t\in [0, T_{max})$, where
\begin{equation}\label{Tmax}
T_{max}= \sup \left\{t>0\;:\; \int_0^t \xi(s)e^{-s(B+k)}\;ds<\frac{1}{u_0(x_0)}\right\}.
\end{equation}
Our goal is to show that $T_{max}<\infty$.

First, applying the definition of $T_{max}$ from \eqref{Tmax}
in formula \eqref{u:form} we obtain the following estimate 
$$
u(x,t)\leq \dfrac{e^{-t(B+k)}}{\frac1{u_0(x)}-\frac1{u_0(x_0)}} = 
 \frac{u_0(x_0)u_0(x)e^{-t(B+k)}}{u_0(x_0)-u_0(x)} \qquad \text{for all} 
\quad (x,t)\in \Omega\times [0,T_{max})
$$
which gives inequality \eqref{u:est:blow}.
Next, 
using this estimate of $u(x,t)$ together with 
the inequality $e^{-t(B+k)}\leq 1$ 
we deduce from
equation \eqref{eq2gs} the  differential inequality
$$
\xi_t\geq  -\xi A_0 +B(1-\xi)
\qquad \text{for all} 
\quad t\in  [0,T_{max}),
$$
where the constant $A_0$ is defined in \eqref{as:u0:2}.
This inequality for $\xi(t)$
 implies that
$$
\xi(t)\geq \min \left\{\xi_0, \frac{B}{A_0+B}\right\}
\qquad \text{for all} 
\quad t\in  [0,T_{max}).
$$
Thus, we obtain the lower bound 
$$
\int_0^t \xi(s)e^{-s(B+k)}\;ds \geq 
\frac{1-e^{-t(B+k)}}{B+k}\min \left\{\xi_0, \frac{B}{A_0+B}\right\},
$$
where the  right-hand side is equal to $1/u_0(x_0)$ for some $t_0>0$
under  assumption \eqref{as:xi:blow}. 
In particular, the denominator of the fraction in
 \eqref{u:form} is equal to zero at $x=x_0$ and some $t_1\leq t_0$ and
this completes the proof that  $T_{max}<\infty$.
\end{proof}

\begin{rem}

In particular, we provide  an example, for which a nonlocal (long-range) ``diffusion''   leads to a blow-up of space heterogeneous solutions. 
In this way, we identify a large class of models with the {\it diffusion induced blow-up}
 in the same spirit as {\it e.g.}~in the works
\cite{PS97,MNY98}; see also 
the review article \cite{FN05}
and the chapter \cite[Ch.~33.2]{SQ07} for other references.

\end{rem}


\section*{Acknowledgments}
A.~Marciniak-Czochra and S. H\"arting were supported by European Research Council Starting Grant No 210680 ``Multiscale mathematical modelling of dynamics of structure formation in cell systems'' and Emmy Noether Programme of German Research Council (DFG). 
G.~Karch was  supported 
by the NCN grant No.~2013/09/B/ST1/04412. %
K.~Suzuki  acknowledges JSPS the Grant-in-Aid for Scientific Research (C) 26400156.

\end{document}